\documentclass{amsart}
\usepackage[noadjust]{cite}

\newtheorem{theorem}{Theorem}[section]
\newtheorem{lemma}[theorem]{Lemma}
\newtheorem{corollary}[theorem]{Corollary}
\newtheorem{proposition}[theorem]{Proposition}
\usepackage{graphicx}
\usepackage{color}
\theoremstyle{definition}

\theoremstyle{remark}
\newtheorem{remark}[theorem]{Remark}

\numberwithin{equation}{section}

\DeclareMathOperator*{\sn}{sn}
\DeclareMathOperator*{\cn}{cn}
\DeclareMathOperator*{\dn}{dn}
\DeclareMathOperator*{\sign}{sign}

\DeclareMathOperator*{\argmin}{arg\,min}

\newcommand{\uu}{v_{2k-1}}
\newcommand{\ww}{v_{2k}}
\DeclareMathOperator*{\Res}{Res}

\renewcommand{\Re}{\operatorname{Re}}
\renewcommand{\Im}{\operatorname{Im}}

\newcommand{\textoverline}[1]{$\overline{\mbox{#1}}$}

\newtheorem{innercustomproblem}{Problem}
\newenvironment{customproblem}[1]
  {\renewcommand\theinnercustomproblem{#1}\innercustomproblem}
  {\endinnercustomproblem}

\begin{document}

\title{Zolotarev's fifth and sixth problems}


\author{Evan S. Gawlik}

\address{Department of Mathematics, University of Hawai`i at M\textoverline{a}noa}
\email{egawlik@hawaii.edu}
\thanks{EG was supported by NSF grants DMS-1703719 and DMS-2012427}

\author{Yuji Nakatsukasa}
\address{Mathematical Institute, University of Oxford, and National Institute of Informatics}

\newcommand{\rr}[1]{\textcolor{red}{#1}}
\newcommand{\rrr}[2]{\textcolor{blue}{(#1)}\textnormal{$\leftarrow$}\textcolor{red}{#2}}
\newcommand{\rrrr}[1]{\textcolor{blue}{(#1) \textnormal{$\leftarrow$} (remove)}}

\email{nakatsukasa@maths.ox.ac.uk}
\thanks{}


\date{\today}

\dedicatory{}

\begin{abstract}
In an influential 1877 paper, Zolotarev asked and answered four questions about polynomial and rational approximation. We ask and answer two questions:
what are the best rational approximants $r$ and $s$ to $\sqrt{z}$ and $\mbox{sign}(z)$ on the unit circle (excluding certain arcs near the discontinuities), with the property that $|r(z)|=|s(z)|=1$ for $|z|=1$? 
We show that the solutions to these problems are related to Zolotarev's third and fourth problems in a nontrivial manner. 

\end{abstract}

\maketitle


\section{Introduction}

Nearly 150 years ago, Zolotarev asked and answered four questions from approximation theory~\cite{zolotarev}.  The first two concern polynomial approximation.  The third is equivalent to the fourth, and the fourth concerns the approximation of $\sign(x) = x/\sqrt{x^2}$ by rational functions on $[-1,-\ell] \cup [\ell,1]$, where $\ell \in (0,1)$.  His solutions to these problems have had lasting impact in approximation theory~\cite{gonvcar1969zolotarev,todd1984applications,reddy1978certain,elementselliptic} and numerical analysis~\cite{beckermann2017singular,mysirev,gawlik2018zolotarev,le2000optimal,stefanfeast, kressner2017fast,wachspress2013adi,rubin2019bounding}.

In this paper, we ask and answer two questions that are closely related to Zolotarev's fourth problem: what are the best (in the uniform norm) rational approximants $r$ and $s$ to $\sqrt{z}$ and $\mbox{sign}(z) = z/\sqrt{z^2}$ on the unit circle (excluding certain arcs near the discontinuities), with the property that $|r(z)|=|s(z)|=1$ for $|z|=1$?  We derive explicit solutions to these two problems and show that they are related in a nontrivial manner to the solution of Zolotarev's fourth problem.  We also show a remarkable property of these solutions: composing two best rational approximants of $\sign(z)$ on the unit circle yields a best rational approximant of higher degree.  This phenomenon closely mirrors the behavior of best rational approximants of $\sign(x)$ on $[-1,-\ell] \cup [\ell,1]$~\cite{mysirev,bogatyrev2010chebyshev,bogatyrev2012rational}.  Related composition laws for best rational approximants have appeared in other contexts, such as the approximation of the square root and $p$th root on positive real intervals~\cite{gawlik2018zolotarev,gawlik2018pth,gawlik2019approximating} and the solution of certain extremal problems involving finite Blaschke products~\cite{ng2013polynomials,ng2015chebyshev}.  Some other rational approximation problems on the unit circle have been studied in~\cite{lukashov2009extremal,lukashov2012exact}.

Let us give precise statements of the problems that we study, beginning with some notation.  We say that a rational function $r(z)=p(z)/q(z)$ has type $(m,n)$ if $p$ and $q$ are polynomials of degree at most $m$ and $n$, respectively.   
We use $\mathcal{R}_{m,n}$ to denote the set of rational functions of type $(m,n)$ with complex coefficients, and $\mathcal{R}_{m,n}^{\mathrm{real}}$ to denote the set of rational functions of type $(m,n)$ with real coefficients. We say that $r \in \mathcal{R}_{m,n}$ has \emph{exact type} $(\mu,\nu)$ if, after canceling common factors, $p$ and $q$ have degree exactly $\mu$ and $\nu$, respectively.  
For each $\Theta \in (0,\pi/2)$, we let
\begin{align*}
S_\Theta &= \{e^{i\theta} \colon \theta \in [-2\Theta,2\Theta] \}, \\
T_\Theta &= \{e^{i\theta} \colon \theta \in [-\Theta,\Theta] \cup [\pi-\Theta,\pi+\Theta] \}.
\end{align*}
We address the following rational approximation problems. 

\begin{customproblem}{Z5} \label{problemZ5}
Given $\Theta \in (0,\pi/2)$ and $n \in \mathbb{N}_0$, find the rational function in $\{ r \in \mathcal{R}_{n,n} \colon |r(z)|=1 \text{ on } |z|=1\}$ that minimizes
\[
\max_{z \in S_\Theta} \left|\arg\left( \frac{r(z)}{\sqrt{z}} \right)\right|.
\]
\end{customproblem}

\begin{customproblem}{Z6} \label{problemZ6}
Given $\Theta \in (0,\pi/2)$ and $m \in \mathbb{N}_0$, find the rational function in $\{ r \in \mathcal{R}_{m,m} \colon |r(z)|=1 \text{ on } |z|=1\}$ that minimizes
\[
\max_{z \in T_\Theta} \left|\arg\left( \frac{r(z)}{\sign(z)} \right)\right|.
\]
\end{customproblem}

We labeled the above problems ``Z5'' and ``Z6'' since they are natural follow-ups to Zolotarev's fourth problem, which we label ``Z4''.  
Zolotarev's fourth problem reads as follows (up to a trivial scaling).
\begin{customproblem}{Z4} \label{problemZ4}
Given $\ell \in (0,1)$ and $m \in \mathbb{N}_0$, find the rational function $r \in \mathcal{R}_{m,m}^{\mathrm{real}}$ that minimizes
\[
\max_{x \in [-1,-\ell] \cup [\ell,1]} \left|r(x)-\sign(x)\right|.
\]
\end{customproblem}

We will derive explicit solutions to Problems~\ref{problemZ5} and~\ref{problemZ6} in Section~\ref{sec:solutions}. 
As we shall see, the solutions to problems Z4--Z6 are connected in a nontrivial manner.
 Then we will study their properties in Section~\ref{sec:properties}, including their behavior under composition and their error.

\section{Solutions} \label{sec:solutions}

In this section, we derive explicit solutions to Problems~\ref{problemZ5} and~\ref{problemZ6}.  The solutions, summarized in Theorem~\ref{thm:main}, involve Jacobi's elliptic functions.  We use $\mathrm{sn}(\cdot,\ell)$, $\mathrm{cn}(\cdot,\ell)$, and $\mathrm{dn}(\cdot,\ell)$ to denote Jacobi's elliptic functions with modulus $\ell$, and we use $\ell' = \sqrt{1-\ell^2}$ to denote the modulus complementary to $\ell$.  We denote the complete elliptic integral of the first kind by $K(\ell) = \int_0^{\pi/2} (1-\ell^2 \sin^2\theta)^{-1/2} \, d\theta$.  

\begin{theorem} \label{thm:main}
Let $m,n \in \mathbb{N}_0$ and $\Theta \in (0,\pi/2)$.  Problem \ref{problemZ5} has a unique solution given by
\begin{equation} \label{sqrtapprox_ratio}
r(z) = r_n(z;\Theta) = \prod_{j=1}^n \frac{ 1+a_j z }{ z+a_j },
\end{equation}
where
\begin{equation} \label{aj}
a_j = \left( \frac{\ell \sn\left(\frac{2j-1}{2n+1}K(\ell'),\ell'\right) + \dn\left(\frac{2j-1}{2n+1}K(\ell'),\ell'\right)}{\cn\left(\frac{2j-1}{2n+1}K(\ell'),\ell'\right)} \right)^{2(-1)^{j+n}},
\end{equation}
and $\ell = \cos\Theta$.  Problem~\ref{problemZ6} has two solutions: the function
\begin{equation}  \label{eq:sdef}
s(z) = s_m(z;\Theta) = i^{1-m} \prod_{j=1}^{m} \frac{z- i b_j}{1+i b_j z},  
\end{equation}
and its reciprocal, where
\begin{equation} \label{bj}
b_j = (-1)^{mj} \left( \frac{ \ell\sn\left(\frac{2j-1}{m}K(\ell'),\ell'\right) + \dn\left(\frac{2j-1}{m}K(\ell'),\ell'\right) }{ \cn\left(\frac{2j-1}{m}K(\ell'),\ell'\right) } \right)^{(-1)^j}.
\end{equation}
\end{theorem}

\begin{remark} \label{remark:oddm}
When $m=2n+1$, the functions $s_m(z;\Theta)$ and $r_n(z;\Theta)$ are related to one another.
Using the observation that
\[
b_j =
\begin{cases}
(-1)^j \sqrt{a_j}, &\mbox{ if $m=2n+1$, $j < n+1$, and $n$ is even}, \\
0, &\mbox{ if $m=2n+1$, $j =n+1$, and $n$ is even}, \\
(-1)^{j+1} \sqrt{a_{2n+2-j}}, &\mbox{ if $m=2n+1$, $j > n+1$, and $n$ is even}, \\
(-1)^j / \sqrt{a_j}, &\mbox{ if $m=2n+1$, $j < n+1$, and $n$ is odd}, \\
\infty, &\mbox{ if $m=2n+1$, $j =n+1$, and $n$ is odd}, \\
(-1)^{j+1} / \sqrt{a_{2n+2-j}}, &\mbox{ if $m=2n+1$, $j > n+1$, and $n$ is odd},
\end{cases}
\]
one checks that 
\begin{equation} \label{signapprox_ratio}
s_{2n+1}(z;\Theta)^{(-1)^n} =  z \prod_{j=1}^n \frac{ z^2+a_j }{ 1+a_j z^2 } = \frac{z}{r_n(z^2;\Theta)}.
\end{equation}
In particular, $s_{2n+1}(z;\Theta)$ has exact type $(2n+1,2n)$ when $n$ is even, and it has exact type $(2n,2n+1)$ when $n$ is odd.
On the other hand, $s_{2n}(z;\Theta)$ has exact type $(2n,2n)$.
\end{remark}

\begin{remark}
For $m,n>1$, neither $s_m(z;\Theta)$ nor $r_n(z;\Theta)$ is a finite Blaschke product, since both functions have at least one root outside the unit disk.
\end{remark}

The following identity will play a central role in our proof of Theorem~\ref{thm:main}.
\begin{theorem} \label{thm:rationalelliptic}
Let $\ell \in (0,1)$ and $m \in \mathbb{N}$.  Let $M = K(\ell)/K(\lambda)$, where $\lambda \in (0,1)$ is determined uniquely by the condition that
\begin{equation} \label{Kratio}
\frac{K(\ell)}{K(\ell')} = \frac{K(\lambda)}{mK(\lambda')}
\end{equation}
holds with $\lambda'=\sqrt{1-\lambda^2}$.  Then the function $s(z)$ in \eqref{eq:sdef} can be expressed as
\begin{equation} \label{rationalelliptic}
s(z) = F(x) + i \sign(\Im z)^m G(x), \quad x = \frac{1}{2}(z+z^{-1}),
\end{equation}
where
\begin{align}
F(x) &= F_m(x;\ell) = \lambda \sn\left( \frac{\sn^{-1}(x/\ell,\ell)}{M}, \lambda \right), \label{zolo} \\
G(x) &= G_m(x;\ell) = \dn \left( \frac{\sn^{-1}(x/\ell,\ell)}{M}, \lambda \right). \label{Gdef}
\end{align}
\end{theorem}

The function $F(x)$ appearing above is none other than Zolotarev's classical solution to Problem \ref{problemZ4} on $[-1,-\ell] \cup [\ell,1]$, scaled to have maximum value 1 on $[\ell,1]$~(\cite{zolotarev},~\cite[Sections 50-51]{elementselliptic}):
\[
\frac{2}{1+\lambda} F = \argmin_{r \in \mathcal{R}_{m,m}^{\mathrm{real}}} \max_{x \in [-1,-\ell] \cup [\ell,1]} |r(x)-\sign(x)|.
\]
It is well-known that $F(x)$ is an odd rational function of exact type $(2\lfloor (m-1)/2 \rfloor+1,2\lfloor m/2 \rfloor)$ that is real-valued on $\mathbb{R}$ and oscillates between $\lambda$ and $1$ on $[\ell,1]=[\cos\Theta,1]$, achieving these values at $m+1$ points $\ell=x_0<x_1< \dots < x_m=1$ in an alternating fashion~(\cite[p. 9]{beckermann2017singular},~\cite[Sections 50-51]{elementselliptic}).  In particular, $F(\ell)=\lambda$.  
Since $|s(z)|=1$ for $|z|=1$, it follows from~(\ref{rationalelliptic}) that $\arg (s(e^{i\theta}))$ equioscillates $m+1$ times on $[-\Theta,\Theta]$, taking values in $[-\arccos\lambda,\arccos\lambda]$.  That is,
\begin{equation} \label{sequi}
\arg(s(e^{i\theta_j})) = \sigma (-1)^j \max_{\theta \in [-\Theta,\Theta]} |\arg(s(e^{i\theta}))| = \sigma (-1)^j \arccos\lambda, \quad j=0,1,\dots,m,
\end{equation}
where $\sigma \in \{-1,1\}$ and 
\[
\theta_j = 
\begin{cases}
-\arccos x_{2j}, &\mbox{ if } j \le m/2, \\
\arccos x_{2m-2j}, &\mbox{ if } j > m/2.
\end{cases}
\]
We will eventually use this fact, together with Remark~\ref{remark:oddm}, to prove the optimality of $s$ and $r$.

\subsection{Proof of Theorem~\ref{thm:rationalelliptic}} \label{sec:proof1}

Let us first prove Theorem~\ref{thm:rationalelliptic}, beginning with the case in which $m=2n+1$ and $n$ is even. 
\medskip
\paragraph{\textbf{Case 1} ($m=2n+1$, $n$ even).}
The fact that the right-hand side of~(\ref{rationalelliptic}) is a rational function of $z$, much less of type $(2n+1,2n)$, is not obvious at first glance.  To prove this, we recall the identities~\cite[p. 214]{elementselliptic}
\begin{align} 
\sn\left( \frac{u}{M}, \lambda \right) &= \frac{\sn(u,\ell)}{M} \prod_{k=1}^n \frac{ 1+\sn^2(u,\ell)\frac{\cn^2(\ww,\ell')}{\sn^2(\ww,\ell')} }{ 1+\sn^2(u,\ell)\frac{\cn^2(\uu,\ell')}{\sn^2(\uu,\ell')} }, \label{snrat} \\
\dn\left( \frac{u}{M}, \lambda \right) &=  \dn(u,\ell) \prod_{k=1}^n \frac{ 1-\sn^2(u,\ell)\dn^2(\uu,\ell') }{ 1+\sn^2(u,\ell)\frac{\cn^2(\uu,\ell')}{\sn^2(\uu,\ell')} }, \label{dnrat}
\end{align}
where $v_j = \frac{j}{m}K(\ell')$.
Let us denote
\begin{align}
\widetilde{F}(z) &= \widetilde{F}_{2n+1}(z;\Theta) =  F_{2n+1}\left(\frac{1}{2}(z+z^{-1})
; \ell \right), \label{Ftilde} \\
\widetilde{G}(z) &= \widetilde{G}_{2n+1}(z;\Theta) = \sign(\Im z) G_{2n+1}\left(\frac{1}{2}(z+z^{-1}); \ell \right). \label{Gtilde}
\end{align}
Note that $\widetilde{F}(z) - i\widetilde{G}(z) = (\widetilde{F}(z) + i\widetilde{G}(z))^{-1}$ since $\lambda^2 \sn^2(\cdot,\lambda) + \dn^2(\cdot,\lambda) = 1$.
Using the fact that 
\[
\ell \sn(u,\ell) = \frac{1}{2}(z+z^{-1}) \iff \dn(u,\ell) = \frac{1}{2i}(z-z^{-1}) \sign(\Im z),
\]
we can write
\begin{align*}
\widetilde{F}(z) &= \frac{\lambda}{2M\ell}(z+z^{-1}) \prod_{k=1}^n \frac{ 1+\frac{\left(\frac{1}{2}(z+z^{-1})\right)^2 \cn^2(\ww,\ell')}{\ell^2 \sn^2(\ww,\ell')} }{ 1+\frac{\left(\frac{1}{2}(z+z^{-1})\right)^2 \cn^2(\uu,\ell')}{\ell^2 \sn^2(\uu,\ell')} }, \\
\widetilde{G}(z) &= \frac{1}{2i} (z-z^{-1})  \prod_{k=1}^n \frac{ 1-\frac{\left(\frac{1}{2}(z+z^{-1})\right)^2}{\ell^2}\dn^2(\uu,\ell') }{ 1+\frac{\left(\frac{1}{2}(z+z^{-1})\right)^2 \cn^2(\uu,\ell')}{\ell^2 \sn^2(\uu,\ell')} }.
\end{align*}
From these expressions it is easy to deduce that $\widetilde{F}(z)+ i \widetilde{G}(z)$ is a rational function which is ostensibly of type $(4n+2,4n+1)$.  However, this turns out to be an overestimate: $\widetilde{F}(z)$ and $i\widetilde{G}(z)$ have $2n+1$ coincident poles (one of which is at $z=0$) with opposite residues, rendering $\widetilde{F}(z)+ i \widetilde{G}(z)$ of type $(2n+1,2n)$.

To see why, it is helpful to rewrite $F(x)$ and $G(x)$ in terms of the Gr\"otsch ring function
\[
\mu(\lambda) = \frac{\pi}{2}\frac{K(\lambda')}{K(\lambda)}, \quad \lambda' = \sqrt{1-\lambda^2}
\]
and the functions
\[
\begin{split}
f_\nu(x) &= \ell \sn(K(\ell)x, \ell),  \\
g_\nu(x) &= \dn(K(\ell)x, \ell),
\end{split}
\quad\quad\quad \ell = \mu^{-1}(1/\nu).
\]
One readily checks, using~(\ref{Kratio}), that
\begin{equation} \label{FGcomp}
\begin{split}
F(x) = f_{m\nu}(f_\nu^{-1}(x)), \\
G(x) = g_{m\nu}(f_\nu^{-1}(x)),
\end{split}
\quad\quad\quad \nu = \frac{1}{\mu(\ell)}.
\end{equation}
Similar formulas for $F$ appear in~\cite{bogatyrev2010chebyshev,bogatyrev2012rational}.

Next, we recall that the poles of $\sn(u,\lambda)$ occur at $u \in \{ 2p K(\lambda) + i (2j-1) K(\lambda') \mid p,j \in \mathbb{Z}\}$~\cite[Equation 2.2.9]{lawden1989elliptic}.  The finite nonzero poles of $\widetilde{F}(z)$ thus occur at those $z \in \mathbb{C}$ for which
\begin{equation} \label{polerelation}
K(\lambda) f_\nu^{-1}\left( \frac{1}{2}(z+z^{-1}) \right) = 2p K(\lambda) + i (2j-1) K(\lambda'), \quad p,j \in \mathbb{Z}.
\end{equation}
That is,
\begin{align*}
\frac{1}{2}(z+z^{-1}) 
&= f_\nu \left( 2p + i (2j-1) \frac{K(\lambda')}{K(\lambda)} \right) \\
&= f_\nu \left( 2p + i \frac{2j-1}{m} \frac{K(\ell')}{K(\ell)} \right) \\
&= \ell \sn\left( 2p K(\ell) +  i \frac{2j-1}{m} K(\ell'), \ell \right) \\
&= (-1)^p \ell \sn( i v_{2j-1}, \ell).
\end{align*}
Here, we used~(\ref{Kratio}), the notation $v_j = \frac{j}{m}K(\ell')$, and the half-period identity $\sn(2pK(\ell) + u, \ell) = (-1)^p \sn(u,\ell)$~\cite[Equation 2.2.11]{lawden1989elliptic}.

The numbers $z$ satisfying $\frac{1}{2}(z+z^{-1}) = (-1)^p \ell \sn(i v_{2j-1},\ell)$ are given by
\[
z = (-1)^p \left( \ell \sn( i v_{2j-1}, \ell) \pm i \dn( i v_{2j-1}, \ell) \right).
\]
Indeed, since $\ell^2 \sn^2(\cdot,\ell) + \dn^2(\cdot,\ell) = 1$, we have $z^{-1} = (-1)^p \left( \ell \sn( i v_{2j-1}, \ell) \mp i \dn( i v_{2j-1}, \ell) \right)$.

We conclude that the finite nonzero poles of $\widetilde{F}(z)$ occur at
\[
\{z_{j,p,q} \mid p,q =0,1, \, j =1,2,\dots,n \},
\]
where
\[
z_{j,p,q} = (-1)^p \left( \ell \sn( i v_{2j-1}, \ell) + (-1)^q i \dn( i v_{2j-1}, \ell) \right).
\]
The finite nonzero poles of $i\widetilde{G}(z)$ are identical, since $\dn(\cdot,\lambda)$ and $\sn(\cdot,\lambda)$ have the same poles.  All of these poles are simple poles thanks to the simplicity of the poles of $\sn$ and $\dn$.

Below we relate the residues of $\widetilde{F}(z)$ to those of $i\widetilde{G}(z)$.

\begin{lemma}
We have
\begin{equation}
\Res(\widetilde{F}, \, z_{j,p,q}) = 
\begin{cases}
\Res( i \widetilde{G}, \, z_{j,p,q}), &\mbox{ if } j+q \text{ is odd, } \\
-\Res( i \widetilde{G}, \, z_{j,p,q}), &\mbox{ if } j+q \text{ is even. } \\
\end{cases}
\end{equation}
In particular, $\Res(\widetilde{F}+i\widetilde{G}, \, z_{j,p,q}) = 0$ if $j+q$ is even.
\end{lemma}
\begin{proof}
In view of~(\ref{FGcomp}), the residues of $F\left(\frac{1}{2}(z+z^{-1})\right)$ and $G\left(\frac{1}{2}(z+z^{-1})\right)$ at $z_{j,p,q}$ are proportional to the residues of $f_{m\nu}(u) = \lambda \sn(K(\lambda)u,\lambda)$ and $g_{m\nu}(u) = \dn(K(\lambda)u,\lambda)$ at $u = f_\nu^{-1}\left(\frac{1}{2}(z_{j,p,q}+z_{j,p,q}^{-1})\right) =: u_{j,p,q}$, with the constant of proportionality the same in both cases.   From~(\ref{polerelation}), we have
\[
K(\lambda) u_{j,p,q} = 2pK(\lambda) + i(2j-1)K(\lambda'),
\] 
so~\cite[p. 41-42]{lawden1989elliptic}
\begin{align}
\Res(\lambda \sn(K(\lambda)u,\lambda), \, u_{j,p,q}) = (-1)^p / K(\lambda), \label{resF} \\
\Res(\dn(K(\lambda)u,\lambda), \, u_{j,p,q}) = (-1)^j i / K(\lambda). 
\end{align}
Since $\sign(\Im z_{j,p,q}) = (-1)^{p+q}$, it follows that 
\begin{equation} \label{resG}
i \sign(\Im z_{j,p,q}) \Res(\dn(K(\lambda)u,\lambda), \, u_{j,p,q}) = (-1)^{j+p+q+1} / K(\lambda).
\end{equation}
Comparing~(\ref{resF}) with~(\ref{resG}), we see that the residues of $\widetilde{F}$ and $i\widetilde{G}$ are equal if $j+q$ is odd, and they are opposite if $j+q$ is even.
\end{proof}

We conclude that the function $\widetilde{F}(z)+i\widetilde{G}(z)$ has only $2n$ finite nonzero poles,
\[
\pm \left( \ell \sn( i v_{2j-1}, \ell) + (-1)^{j+1} i \dn( i v_{2j-1}, \ell) \right), \quad j=1,2,\dots,n.
\]
All of these poles are simple.  Since $\widetilde{F}(z)+i\widetilde{G}(z)$ has unit modulus on the unit circle, its finite nonzero roots are the reciprocals of these poles.  

Now observe that since $\sn(iu,\ell) = i\frac{\sn(u,\ell')}{\cn(u,\ell')}$ and $\dn(iu,\ell) = \frac{\dn(u,\ell')}{\cn(u,\ell')}$~\cite[Equation 2.6.12]{lawden1989elliptic}, we have
\begin{align}
\left( \ell \sn( i v_{2j-1}, \ell) + (-1)^{j+1} i \dn( i v_{2j-1}, \ell) \right)^2 \nonumber
&= \left( \ell \sn( i v_{2j-1}, \ell) + i \dn( i v_{2j-1}, \ell) \right)^{2(-1)^{j+1}} \nonumber \\
&= \left( \frac{ i \ell \sn( v_{2j-1}, \ell' ) + i \dn( v_{2j-1}, \ell' )}{ \cn( v_{2j-1}, \ell' ) } \right)^{2(-1)^{j+1}} \nonumber \\
&= -\left( \frac{ \ell \sn( v_{2j-1}, \ell' ) + \dn( v_{2j-1}, \ell' )}{ \cn( v_{2j-1}, \ell' ) } \right)^{2(-1)^{j+1}} \nonumber \\
&= -1/a_j, \label{ajinv}
\end{align}
where $a_j$ is given by~(\ref{aj}) (recall that we are still focusing on the case in which $m=2n+1$ and $n$ is even).

It follows that 
\begin{equation} \label{rationalguess}
\widetilde{F}(z) + i \widetilde{G}(z)
= e^{i\alpha} z^k \prod_{j=1}^n \frac{z^2+a_j}{1+a_j z^2}
\end{equation}
for some $\alpha \in \mathbb{R}$ and some $k \in \mathbb{Z}$.  We must have $e^{i\alpha}=1$ since $\widetilde{F}(1) = F(1)>0$ and $\widetilde{G}(1)=0$.  We must have $k \ge -1$ since $\widetilde{F}(z)$ and $\widetilde{G}(z)$ each have simple poles at $z=0$.  We must have $k \le 1$ for a similar reason: $\widetilde{F}(z) - i\widetilde{G}(z) = \frac{1}{\widetilde{F}(z) + i\widetilde{G}(z)}$ cannot have a pole of order $>1$ at $z=0$.  To conclude, note that at $z=i$, the left-hand side of~(\ref{rationalguess}) evaluates to $i$, while the right-hand side evaluates to $i^k (-1)^n = i^k$.  The only possibility is $k=1$.  Thus,
\[
\widetilde{F}(z)+i\widetilde{G}(z) = z \prod_{j=1}^n \frac{z^2+a_j}{1+a_j z^2} = \frac{z}{r_n(z^2;\Theta)}, \quad \text{ if $m=2n+1$ and $n$ is even. }
\]
In view of~(\ref{signapprox_ratio}), this completes the proof of Theorem~\ref{thm:rationalelliptic} for the case in which $m=2n+1$ and $n$ is even.

\medskip
\paragraph{\textbf{Case 2} ($m=2n+1$, $n$ odd).}
The case in which $m=2n+1$ and $n$ is odd is handled similarly.  This time,~(\ref{ajinv}) becomes
\[
\left( \ell \sn( i v_{2j-1}, \ell) + (-1)^{j+1} i \dn( i v_{2j-1}, \ell) \right)^2 = -a_j,
\]
so that~(\ref{rationalguess}) becomes
\begin{equation} \label{rationalguess2}
\widetilde{F}(z) + i \widetilde{G}(z)
= e^{i\alpha} z^k \prod_{j=1}^n \frac{1+a_j z^2}{z^2+a_j}.
\end{equation}
As before, we can argue that $e^{i\alpha}=1$ and $-1 \le k \le 1$.  At $z=i$, the left-hand side evaluates to $i$, while the right-hand side evaluates to $i^k (-1)^n = -i^k$.  We conclude that $k=-1$.  That is,
\[
\widetilde{F}(z)+i\widetilde{G}(z) = \frac{1}{z} \prod_{j=1}^n \frac{1+a_j z^2}{z^2+a_j} = \frac{r_n(z^2;\Theta)}{z}, \quad \text{ if $m=2n+1$ and $n$ is odd. }
\]

\medskip
\paragraph{\textbf{Case 3} ($m=2n$).}
Finally, when $m=2n$, the identities~(\ref{snrat}-\ref{dnrat}) change to~\cite[p. 214]{elementselliptic}
\begin{align} 
\sn\left( \frac{u}{M}, \lambda \right) &= \frac{\sn(u,\ell)}{M} \frac{ \prod_{k=1}^{n-1} 1+\sn^2(u,\ell)\frac{\cn^2(\ww,\ell')}{\sn^2(\ww,\ell')} }{ \prod_{k=1}^n 1+\sn^2(u,\ell)\frac{\cn^2(\uu,\ell')}{\sn^2(\uu,\ell')} }, \label{snrat2} \\
\dn\left( \frac{u}{M}, \lambda \right) &=  \prod_{k=1}^n \frac{ 1-\sn^2(u,\ell)\dn^2(\uu,\ell') }{ 1+\sn^2(u,\ell)\frac{\cn^2(\uu,\ell')}{\sn^2(\uu,\ell')} }. \label{dnrat2}
\end{align}
Note that in contrast to~\cite[p. 214]{elementselliptic}, we terminated the product in the numerator of~(\ref{snrat2}) at $k=n-1$ rather than $k=n$ since $\cn(v_{2n},\ell')=0$ when $m=2n$.
Accordingly, we put
\begin{align}
\widetilde{F}(z) &= \widetilde{F}_{2n}(z;\Theta) = F_{2n}\left(\frac{1}{2}(z+z^{-1}); \ell\right), \label{Ftilde2} \\
\widetilde{G}(z) &= \widetilde{G}_{2n}(z;\Theta) =  G_{2n}\left(\frac{1}{2}(z+z^{-1}); \ell\right), \label{Gtilde2}
\end{align}
and we observe that $\widetilde{F}(z) + i\widetilde{G}(z)$ is a rational function which is ostensibly of type $(4n,4n)$.  However, $2n$ of the poles $z_{j,p,q}$ coincide and have opposite residues; this time it is those poles $z_{j,p,q}$ for which $j+p$ is even, since the factor $\sign(\Im z_{j,p,q})$ does not appear in the analysis (compare~(\ref{Gtilde2}) with~(\ref{Gtilde})).
Since $\widetilde{F}$ and $i\widetilde{G}$ have $2n$ coincident poles with opposite residues, $\widetilde{F}(z)+i\widetilde{G}(z)$ is in fact of type $(2n,2n)$.  The poles of $\widetilde{F}(z) + i\widetilde{G}(z)$ are
\[\begin{split}
(-1)^{j+1} \left( \ell \sn( i v_{2j-1}, \ell) \pm i \dn( i v_{2j-1}, \ell) \right)  &= \pm (-1)^{j+1} i b_j^{\pm (-1)^j} \\&= -(ib_j)^{\pm(-1)^j} , \quad j=1,2,\dots,n,
\end{split}\]
where $b_j$ is given by~(\ref{bj}).  One checks that the sets $\{-(ib_j)^{(-1)^j}\}_{j=1}^n \cup \{-(ib_j)^{-(-1)^j}\}_{j=1}^n$ and $\{-1/(ib_j)\}_{j=1}^m$ are equal, so $\widetilde{F}(z)+i\widetilde{G}(z)$ must have the form
\begin{equation} \label{rationalguess3}
\widetilde{F}(z) + i \widetilde{G}(z)
= e^{i\alpha} z^k \prod_{j=1}^m \frac{z - i b_j}{1 + i b_j z}.
\end{equation}
Again, we can argue that $-1 \le k \le 1$, but this time we cannot conclude that $e^{i\alpha}=1$ by evaluating both sides of~(\ref{rationalguess3}) at $z=1$.  Instead, we evaluate both sides at $z=i$ to obtain $i = e^{i\alpha} i^{k+m}$, and we evaluate both sides at $z=-i$ to obtain $i = e^{i\alpha} (-i)^{k+m}$.  We conclude that $k+m$ is even, and since $m$ is too, we have $k=0$ and $e^{i\alpha} = i^{1-m}$.  

This completes the proof of Theorem~\ref{thm:rationalelliptic}.

As a final remark, we note that Theorem~\ref{thm:rationalelliptic} also holds trivially when $m=0$ if we adopt the convention that $\lambda:=0$ when $m=0$.

\subsection{Proof of Theorem~\ref{thm:main}} \label{sec:proof2}

Let us now use Theorem~\ref{thm:rationalelliptic} to prove Theorem~\ref{thm:main}.  We first elaborate on the relation between Problems~\ref{problemZ5} and~\ref{problemZ6}.  Observe that if $w=z^2$, then $w \in S_\Theta \iff z \in T_\Theta$, and
\begin{equation} \label{samearg}
\arg\left( \frac{z p(z^2)/q(z^2) }{ \sign(z) } \right) = -\arg\left( \frac{q(w)/p(w)}{\sqrt{w}} \right) = -\arg\left( \frac{q(z^2)/(zp(z^2)) }{ \sign(z) } \right)
\end{equation}
for any polynomials $p$ and $q$.  In view of~(\ref{sequi}) and Remark~\ref{remark:oddm}, it follows that $\arg\left( \frac{r(e^{i\theta})}{\sqrt{e^{i\theta}}} \right)$ 
equioscillates $2n+2$ times on $[-2\Theta,2\Theta]$, taking values in $[-\arccos\lambda,\arccos\lambda]$.  
Suppose now that $\widetilde{r}(z)$ is another rational function of type $(n,n)$ satisfying $|\widetilde{r}(z)|=1$ for $|z|=1$ and 
\[
\max_{z \in S_\Theta} \left| \arg\left( \frac{\widetilde{r}(z)}{\sqrt{z}} \right) \right| \le \arccos\lambda.
\]
Then the equioscillation of $\arg\left( \frac{r(e^{i\theta})}{\sqrt{e^{i\theta}}} \right)$ implies that on $[-2\Theta,2\Theta]$,
\[
\arg\left( \frac{r(e^{i\theta})}{\sqrt{e^{i\theta}}} \right) - \arg\left( \frac{\widetilde{r}(e^{i\theta})}{\sqrt{e^{i\theta}}} \right) = \arg\left( \frac{r(e^{i\theta})}{\widetilde{r}(e^{i\theta})} \right)
\]
has at least $2n+1$ roots, counted with multiplicity.  
Hence, the numerator of $r(z)-\widetilde{r}(z)$ has at least $2n+1$ roots, counted with multiplicity.
Since $r(z)-\widetilde{r}(z)$ has type $(2n,2n)$, it follows that $\widetilde{r}=r$.  
This shows that Problem~\ref{problemZ5} has a unique solution, namely $r$.

The proof that Problem~\ref{problemZ6} has precisely two solutions---$s(z)$ and $s(z)^{-1}$---proceeds similarly.  Assume $m>0$; otherwise the claim is trivial.  We see from~(\ref{sequi}) that 
\[
\max_{z \in T_\Theta} \left| \arg\left( \frac{s(z)}{\sign(z)} \right) \right| = \arccos\lambda,
\]
and $\arg\left( \frac{s(e^{i\theta})}{\sign(e^{i\theta})} \right)$ equioscillates $m+1$ times on $[-\Theta,\Theta]$ and $m+1$ times on $[\pi-\Theta,\pi+\Theta]$, owing to the fact that $\sign(e^{i\theta}) = 1$ when $\theta \in [-\Theta,\Theta]$, $\sign(e^{i\theta}) = -1$ when $\theta \in [\pi-\Theta,\pi+\Theta]$, and $-s(e^{i\theta}) = s(e^{i(\pi-\theta)})^{-1}$ for all $\theta$.  The same is true of $\arg\left( \frac{s(e^{i\theta})^{-1}}{\sign(e^{i\theta})} \right)$ since 
\[
\arg\left( \frac{s(e^{i\theta})^{-1}}{\sign(e^{i\theta})} \right) = -\arg\left( \frac{s(e^{i\theta})}{\sign(e^{i\theta})} \right)
\]
for all $\theta$.  Suppose now that $\widetilde{s}(z)$ is another rational function of type $(m,m)$ satisfying $|\widetilde{s}(z)|=1$ for $|z|=1$ and 
\[
\max_{z \in T_\Theta} \left| \arg\left( \frac{\widetilde{s}(z)}{\sign{z}} \right) \right| \le \arccos\lambda.
\]
Then the same reasoning as above shows that the numerator of $\widetilde{s}(z)-s(z)$ has at least $2m$ roots counted with multiplicity.  At least $m$ of them lie in $\{z \in T_\Theta \mid \Re z > 0\}$, and at least $m$ of them lie in $\{z \in T_\Theta \mid \Re z < 0\}$.  Likewise, the numerator of $\widetilde{s}(z)-s(z)^{-1}$ has at least $2m$ roots counted with multiplicity, at least $m$ of which lie in $\{z \in T_\Theta \mid \Re z > 0\}$ and at least $m$ of which lie in $\{z \in T_\Theta \mid \Re z < 0\}$.  By considering the graphs of $\arg s(z)$ and $\arg (s(z)^{-1})$ (see Figure~\ref{fig:argsign}), there must also be at least one point $z \in \{z \in \mathbb{C} \mid |z|=1, \, z \notin \operatorname{int}(T_\Theta)\}$ where either $\widetilde{s}(z)=s(z)$ or $\widetilde{s}(z)=s(z)^{-1}$.  (If all such points happen to be on the boundary of $T_\Theta$, then it is easy to see that there must have been more than $2m$ points in $T_\Theta$ (counting multiplicities) where either $\widetilde{s}(z)=s(z)$ or $\widetilde{s}(z)=s(z)^{-1}$ to begin with.)  We conclude that either $\widetilde{s}=s$ or $\widetilde{s}=1/s$.  This shows that Problem~\ref{problemZ6} has precisely two solutions: $s$ and $1/s$.

\begin{figure} 
\hspace{-0.2in}
\includegraphics[scale=0.34]{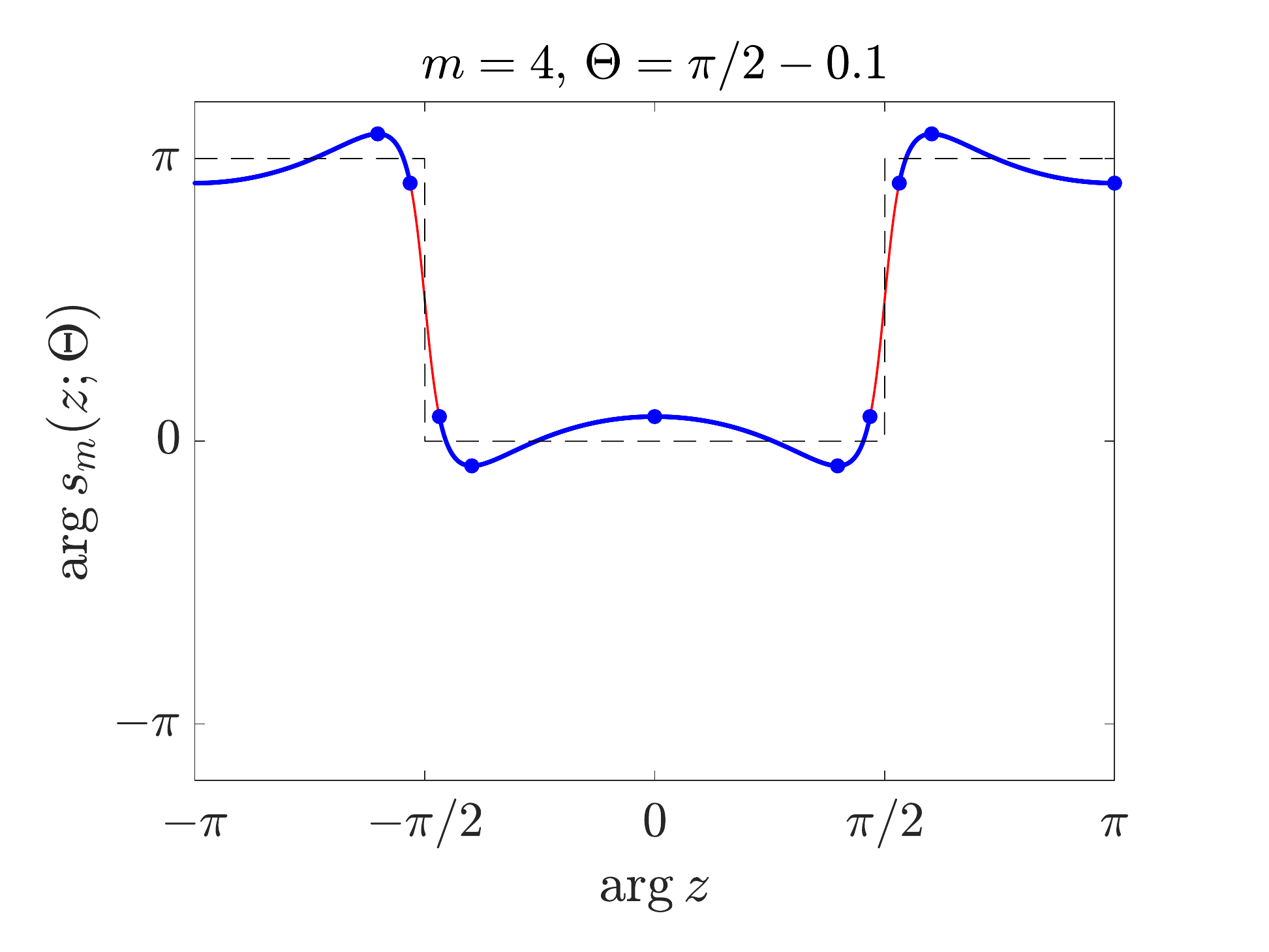}
\hspace{-0.35in}
\includegraphics[scale=0.34]{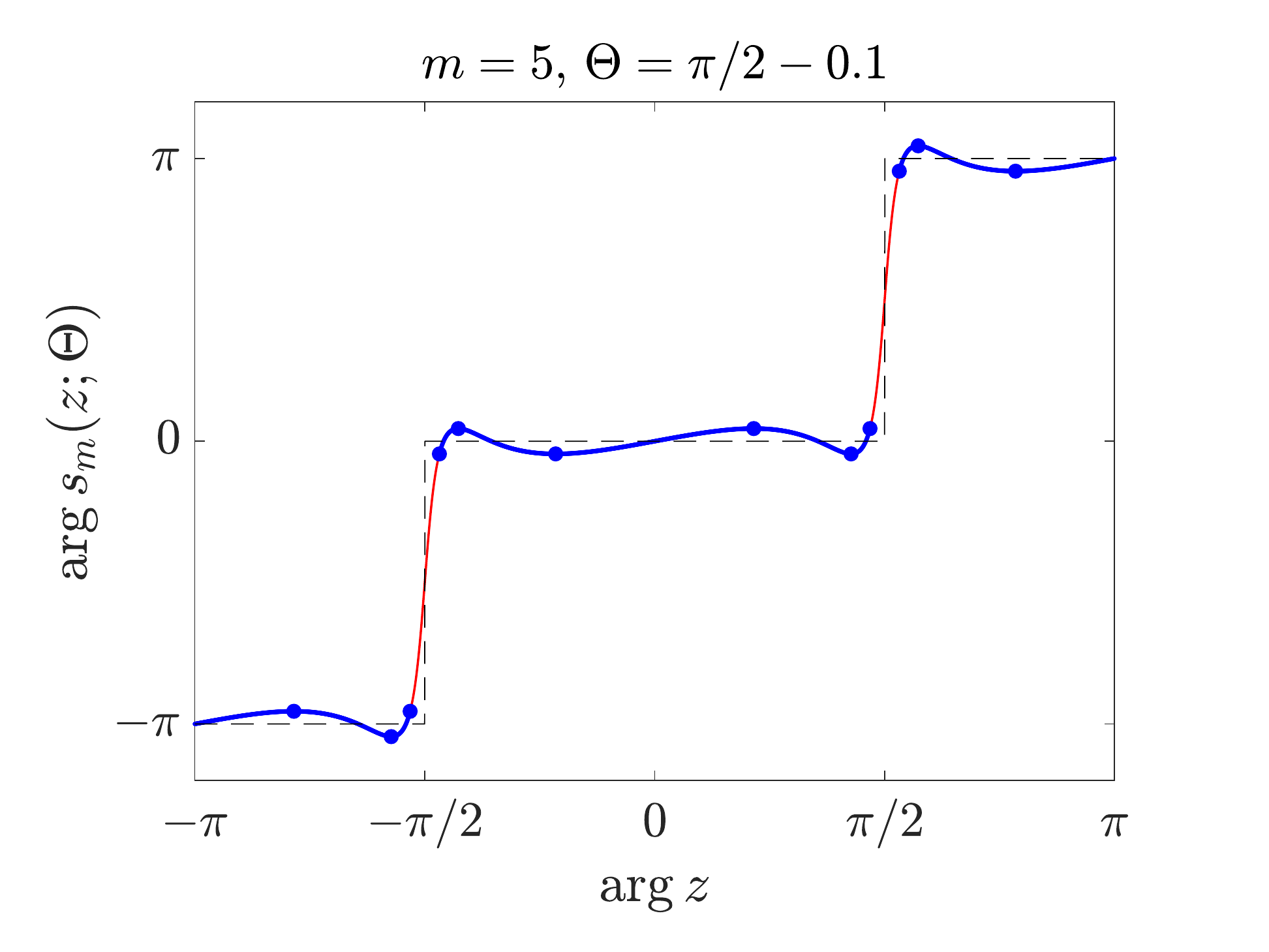}
\caption{Plots of $\arg s_m(z;\Theta)$ with $\Theta = \pi/2-0.1$ and $m=4,5$.  Portions of the graph corresponding to points $z \in T_\Theta$ (respectively, $z \notin T_\Theta$) are colored blue (respectively, red).  Extrema of the error on $T_\Theta$ are marked with blue dots.   The dashed line is $\arg \sign z$.  Both the horizontal and vertical axes are to be interpreted modulo $2\pi$.  The graphs of $\arg \left( s_m(z;\Theta)^{-1} \right)$ are obtained by reflecting the above graphs across the horizontal axis.}
\label{fig:argsign}
\end{figure}

\section{Properties of the Solutions} \label{sec:properties}

In this section, we study the error committed by the functions $r_n(z;\Theta)$ and $s_m(z;\Theta)$ from Theorem~\ref{thm:main}, and we study the behavior of $r_n(z;\Theta)$ and $s_m(z;\Theta)$ under composition.

\subsection{Error} \label{sec:error}

To study the error, we appeal to well-known properties of the function $F_m(x;\ell)$ defined in~(\ref{zolo}).  As we noted earlier, $\frac{2}{1+\lambda}F_m(x;\ell)$ is the solution to Problem \ref{problemZ4} on $[-1,-\ell] \cup [\ell,1]$.

The number $\frac{1-\lambda}{1+\lambda} = \max_{x \in [-1,-\ell] \cup [\ell,1]} \left|\frac{2}{1+\lambda}F_m(x;\ell)-\operatorname{sign}(x)\right|$ is well-studied; it satisfies~\cite[p. 9]{beckermann2017singular}
\begin{equation} \label{lambdaZ}
\frac{1-\lambda}{1+\lambda} = \frac{2\sqrt{Z_m}}{1+Z_m},
\end{equation}
where $Z_m = Z_m([-1,-\ell],[\ell,1])$ denotes the \emph{Zolotarev number} of the sets $[-1,-\ell]$ and $[\ell,1]$:
\begin{equation} \label{zolonumber}
Z_m(E,F) = \inf_{r \in \mathcal{R}_{m,m}} \frac{ \sup_{z \in E} |r(z)| }{ \inf_{z \in F} |r(z)| }.
\end{equation}
An explicit formula for $Z_m$ ($m \ge 1$) is~\cite[Theorem 3.1]{beckermann2017singular}
\[
Z_m = 4 \rho^{-2m} \prod_{j=1}^\infty \frac{ (1+\rho^{-8jm})^4 }{ (1+\rho^{4m}\rho^{-8jm})^4 } \le 4 \rho^{-2m},
\]
where
\[
\rho = \exp\left( \frac{ \pi K(\ell) }{ K(\ell') } \right) = \exp\left( \frac{ \pi K(\cos\Theta) }{ K(\sin\Theta) } \right).
\]
Note that the bound $Z_m \le 4\rho^{-2m}$ also obviously holds for $m=0$.
Solving for $\lambda$ in~(\ref{lambdaZ}), we find that
\begin{equation}
\max_{z \in T_\Theta} \left| \arg\left( \frac{s_m(z;\Theta)}{\sign(z)} \right) \right| = \arccos\lambda = \arccos\left( \left( \frac{1-\sqrt{Z_m}}{1+\sqrt{Z_m}} \right)^2 \right).
\end{equation}

We derive upper bounds for this quantity below.
\begin{lemma} \label{lemma:arccos}
For every $x \ge 0$,
\[
\arccos\left( \left( \frac{1-\sqrt{x}}{1+\sqrt{x}} \right)^2 \right) \le 2 \sqrt{2} x^{1/4}.
\]
\end{lemma}
\begin{proof}
Let $f(x) = \arccos\left( \left( \frac{1-\sqrt{x}}{1+\sqrt{x}} \right)^2 \right)$ and $g(x) = 2 \sqrt{2} x^{1/4}$.  Since $f(0)=g(0)=0$ and
\[
f'(x) = \frac{1-\sqrt{x}}{\sqrt{2}x^{3/4}(1+\sqrt{x})\sqrt{1+x}} < \frac{1}{\sqrt{2} x^{3/4}} = g'(x), \quad x > 0,
\]
we have $f(x) = \int_0^x f'(t) \, dt \le \int_0^x g'(t) \, dt = g(x)$ for every $x \ge 0$.
\end{proof}

\begin{theorem}\label{thm:conv}
Let $\Theta \in (0,\pi/2)$ and $m,n \in \mathbb{N}_0$.  We have
\begin{equation} \label{signerr}
\max_{z \in T_\Theta} \left| \arg\left( \frac{s_m(z;\Theta)}{\sign(z)} \right) \right| \le  4 \rho^{-m/2} \le 4 \left[ \exp\left( \frac{\pi^2}{4\log(4\sec\Theta)} \right) \right]^{-m}
\end{equation}
and
\begin{equation} \label{sqrterr}
\max_{z \in S_\Theta} \left| \arg \left( \frac{r_n(z;\Theta)}{\sqrt{z}} \right) \right| \le  4 \rho^{-(n+1/2)} \le 4 \left[ \exp\left( \frac{\pi^2}{2\log(4\sec\Theta)} \right) \right]^{-(n+1/2)}.
\end{equation}
\end{theorem}
\begin{proof}
Using Lemma~\ref{lemma:arccos} and the inequality~\cite[p. 8]{beckermann2017singular}
\[
\frac{\pi}{2} K(\sqrt{1-x^2}) / K(x) \le \log{4/x}, \quad 0 < x < 1,
\]
we compute
\begin{align*}
\max_{z \in T_\Theta} \left| \arg\left( \frac{s_m(z;\Theta)}{\sign(z)} \right) \right|
\le 2\sqrt{2} Z_m^{1/4} 
&\le 2\sqrt{2} \left( 4 \rho^{-2m} \right)^{1/4} 
= 4 \rho^{-m/2} \\
&\quad\quad\le 4 \left[ \exp\left( \frac{\pi^2}{2\log(4\sec\Theta)} \right) \right]^{-m/2}.
\end{align*}
The bound~(\ref{sqrterr}) follows from Remark~\ref{remark:oddm} and~(\ref{samearg}), which imply
\[
\max_{z \in S_\Theta} \left| \arg \left( \frac{r_n(z;\Theta)}{\sqrt{z}} \right) \right| = \max_{z \in T_\Theta} \left| \arg\left( \frac{s_{2n+1}(z;\Theta)}{\sign(z)} \right) \right|.
\]
\end{proof}

Theorem~\ref{thm:conv} is illustrated in Figure~\ref{fig:thmconv}, which shows the bounds are very tight. 

Figure~\ref{fig:contour} plots the absolute errors  $\left| r_n(z;\Theta) - \sqrt{z} \right|$ and $\left| s_m(z;\Theta) - \sign(z) \right|$ for $z \in \mathbb{C}$.

\begin{figure} 
\includegraphics[scale=0.45]{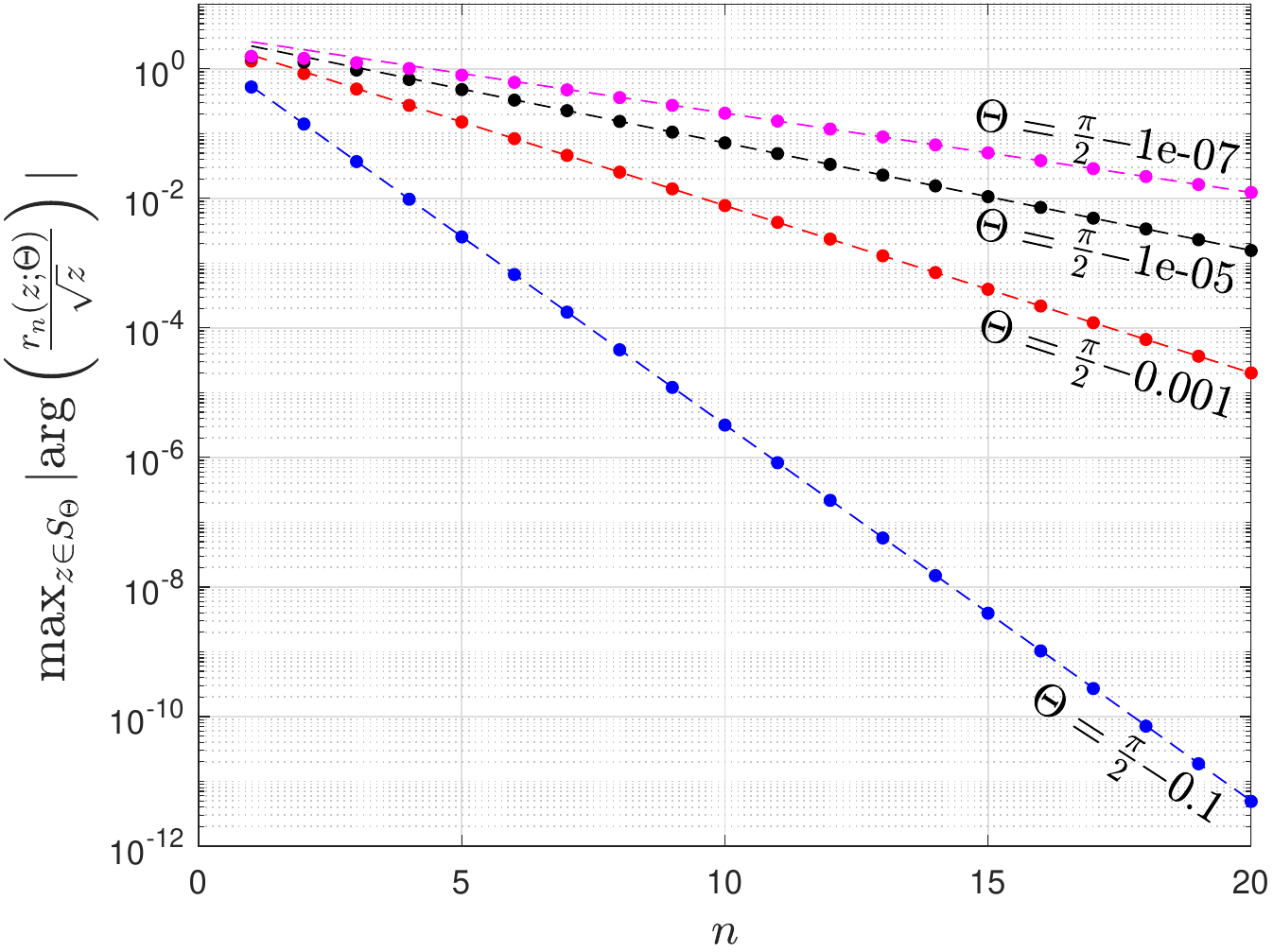}
\hspace{0.05in}
\includegraphics[scale=0.45]{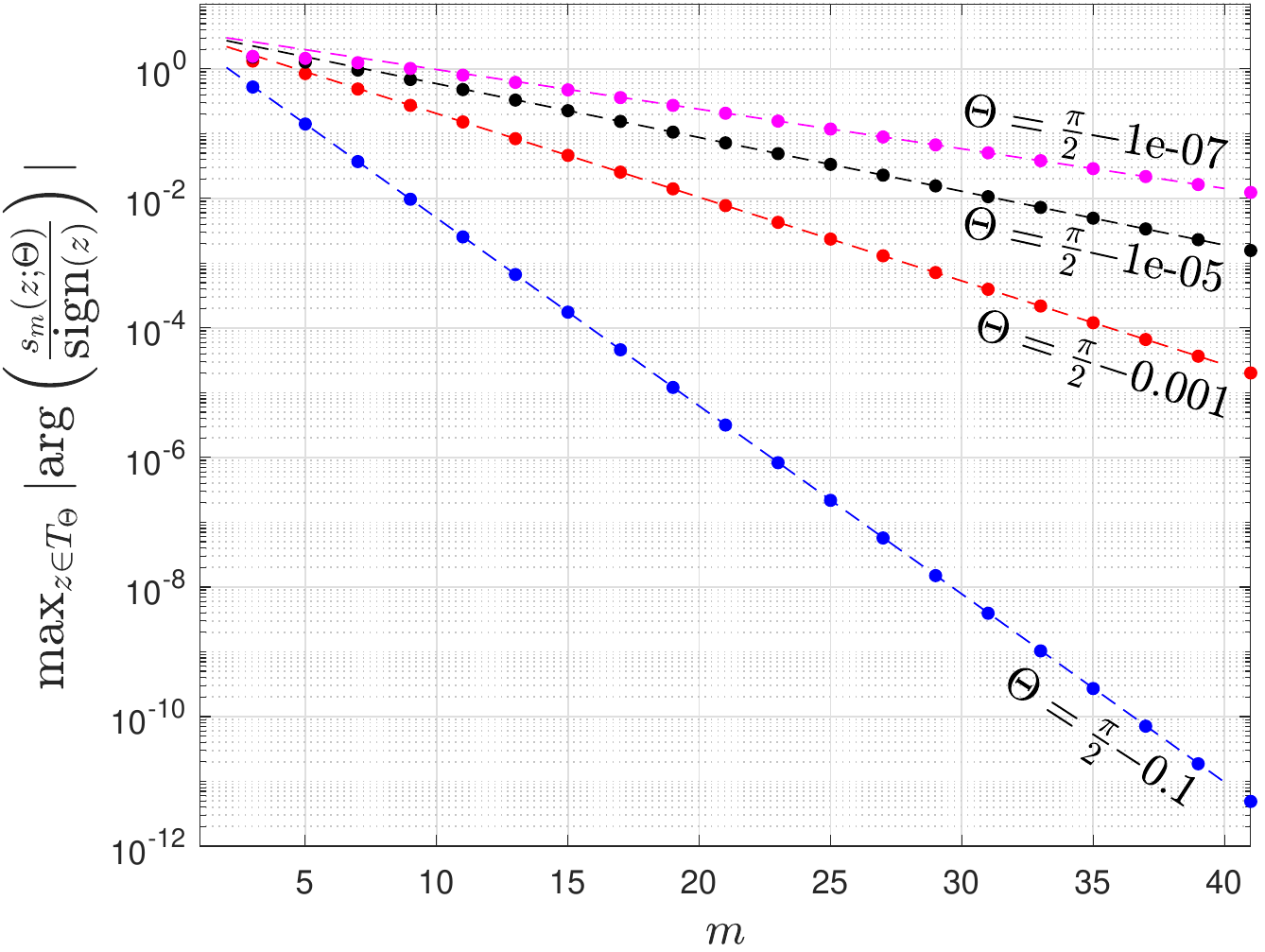}
\caption{The errors (dots) and their bounds 
in Theorem~\ref{thm:conv} (dashed lines, the rightmost bounds in \eqref{signerr}, \eqref{sqrterr} are shown) for Z5 (left) and Z6 (right).
}
\label{fig:thmconv}
\end{figure}

\begin{figure} 
\centering
\includegraphics[scale=0.45]{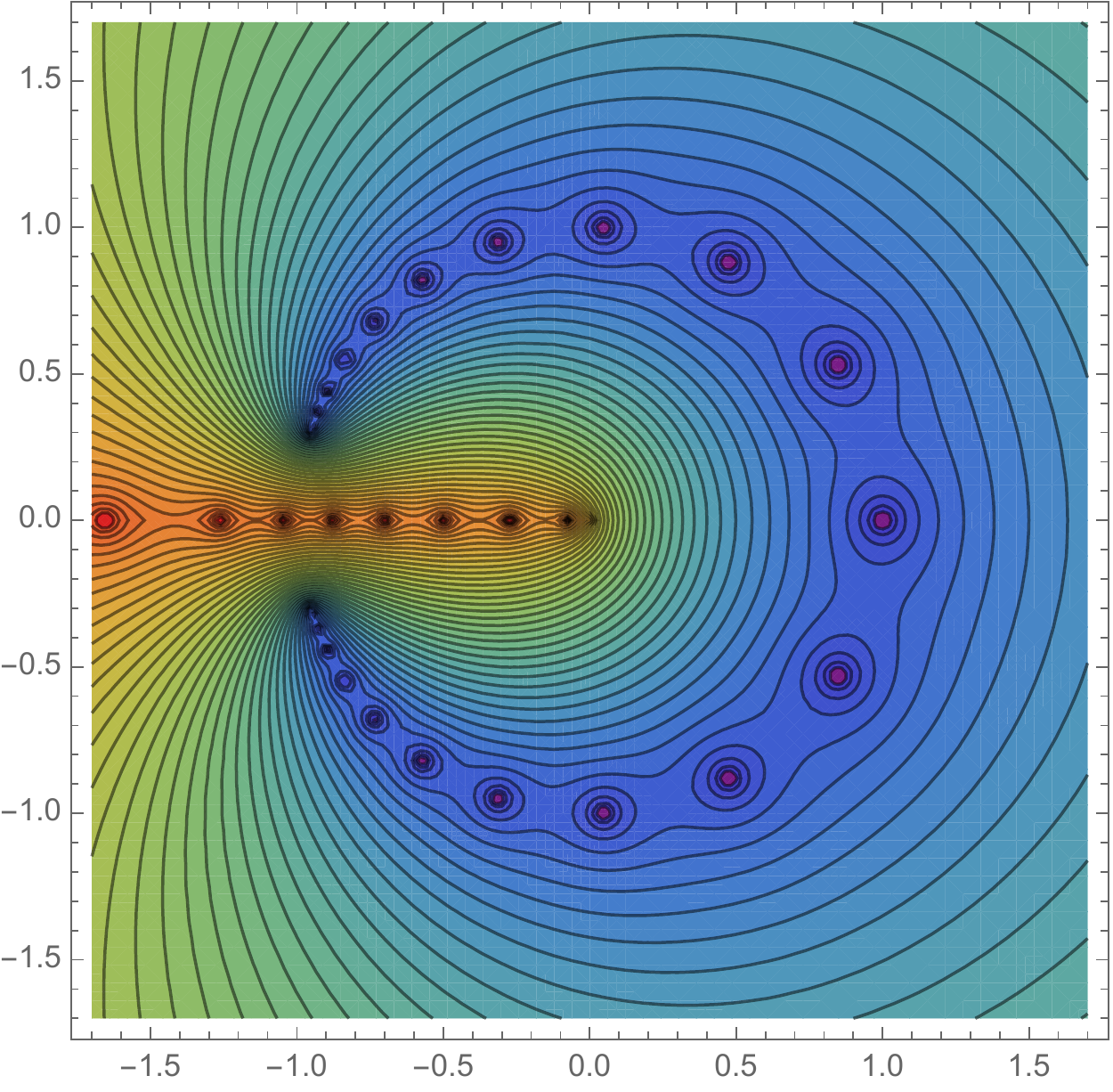}
\includegraphics[scale=0.45]{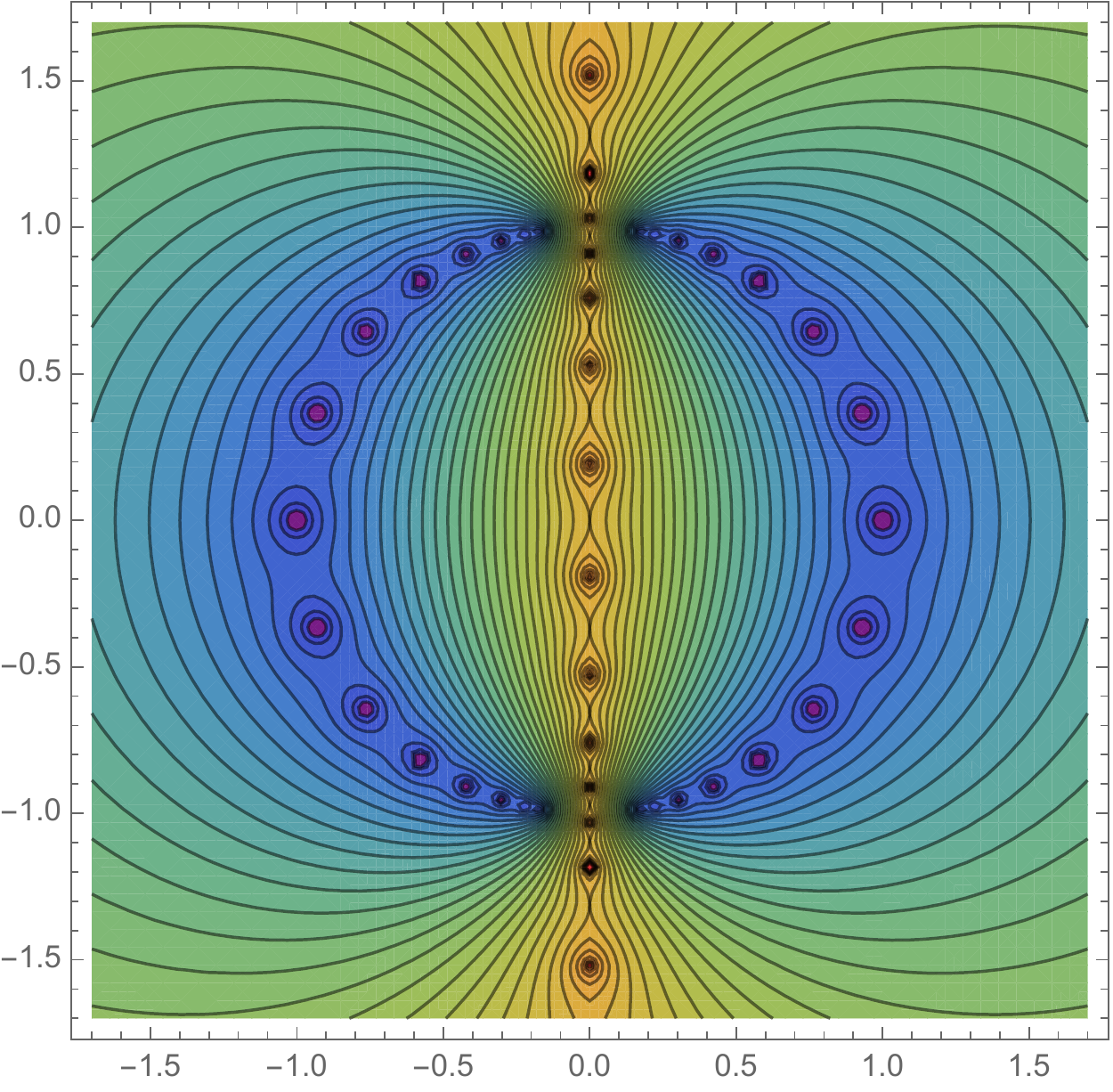}
\caption{Contours of the error $\left| r_n(z;\Theta) - \sqrt{z} \right|$ (left) and $\left| s_m(z;\Theta) - \sign(z) \right|$ (right) in the complex plane with $n=11$, $m=17$, and $\Theta = \frac{\pi}{2}-0.15$. 
The extrema on the unit circle are the zeros of the error, and the extrema on the coordinate axes are poles of the approximants.
}
\label{fig:contour}
\end{figure}

\subsection{Composition}

Next, we show that when two solutions of Problem~\ref{problemZ6} are composed with one another, the resulting function is a solution of Problem~\ref{problemZ6} of higher degree.

\begin{theorem} \label{thm:comp}
Let $\Theta \in (0,\pi/2)$, $m,\widetilde{m} \in \mathbb{N}_0$, and $\widetilde{\Theta} = \left| \arg(s_{m}(e^{i\Theta}; \Theta)) \right|$.  Then
\[
s_{\widetilde{m}}(s_{m}(z;\Theta);\widetilde{\Theta}) = s_{\widetilde{m}m}(z;\Theta).
\]
\end{theorem}
\begin{proof}
This is essentially a consequence of the identities
\begin{align}
f_{\widetilde{m}\widetilde{\nu}} \circ f_{\widetilde{\nu}}^{-1} \circ f_{m\nu} \circ f_\nu^{-1} &= f_{\widetilde{m}m\nu} \circ f_\nu^{-1}, \label{fcomp} \\
g_{\widetilde{m}\widetilde{\nu}} \circ f_{\widetilde{\nu}}^{-1} \circ f_{m\nu} \circ f_\nu^{-1} &= \pm g_{\widetilde{m}m\nu} \circ f_\nu^{-1}, \label{gcomp}
\end{align}
which hold on $[-1,1]$ whenever 
\begin{equation} \label{nurelation}
\widetilde{\nu} = m\nu.
\end{equation}
(The $\pm$ sign in~(\ref{gcomp}) is $+$ at $x$ if $g_{m\nu}(f_\nu^{-1}(x))^{\widetilde{m}}$ is positive and $-$ at $x$ if  $g_{m\nu}(f_\nu^{-1}(x))^{\widetilde{m}}$ is negative, owing to the branch cut structure of $\sn^{-1}$.)

To flesh out the details, note that~(\ref{nurelation}) holds for $\nu = 1/\mu(\ell)$ and $\widetilde{\nu} = 1/\mu(\widetilde{\ell})$ if and only if
\[
\frac{K(\ell)}{K(\ell')} = \frac{K(\widetilde{\ell})}{mK(\widetilde{\ell}')}.
\]
Comparing with~(\ref{Kratio}), we see that this happens precisely when $\widetilde{\ell} = \lambda = F_m(\ell;\ell)$.  
In turn, this holds if and only if $\ell = \cos\Theta$ and $\widetilde{\ell} = \cos\widetilde{\Theta}$ with $\widetilde{\Theta} = \left| \arg(s_m(e^{i\Theta}; \Theta)) \right|$.

Let us now compute $s_{\widetilde{m}}(s_{m}(z;\Theta);\widetilde{\Theta})$ under the assumption that $\widetilde{\Theta} = \left| \arg(s_{m}(e^{i\Theta}; \Theta)) \right|$.  
Since 
\[
s_m(z;\Theta) = \widetilde{F}_m(z;\Theta) + i \widetilde{G}_m(z;\Theta) = f_{m\nu}(f_\nu^{-1}(x)) + i (\sign \Im z)^m g_{m\nu}(f_\nu^{-1}(x))
\] 
and $s_m(z;\Theta)^{-1} = \widetilde{F}_m(z;\Theta) - i \widetilde{G}_m(z;\Theta)$, we have
\[
\frac{1}{2}(s_{m}(z;\Theta) + s_{m}(z;\Theta)^{-1}) = \widetilde{F}_{m}(z;\Theta) = f_{m\nu}(f_\nu^{-1}(x)), 
\]
where
\[
x = \frac{1}{2}(z+z^{-1}), \quad \nu = \frac{1}{\mu(\ell)}, \quad \ell = \cos\Theta.
\]
Thus, denoting 
\begin{align*}
\widetilde{\nu} &= 1/\mu(\cos\widetilde{\Theta}) = m\nu, \\
\sigma &= (\sign \Im s_{m}(z;\Theta))^{\widetilde{m}} = (\sign\Im z)^{\widetilde{m}m} \sign( g_{m\nu}(f_\nu^{-1}(x)) )^{\widetilde{m}}, \\ 
\tau &= \sign( g_{m\nu}(f_\nu^{-1}(x)) )^{\widetilde{m}},
\end{align*}
we find
\begin{align*}
s_{\widetilde{m}}(s_{m}(z;\Theta);\widetilde{\Theta})
&= f_{\widetilde{m}\widetilde{\nu}}( f_{\widetilde{\nu}}^{-1}(f_{m\nu}(f_\nu^{-1}(x)))) + i \sigma g_{\widetilde{m}\widetilde{\nu}}( f_{\widetilde{\nu}}^{-1}(f_{m\nu}(f_\nu^{-1}(x)))) \\
&= f_{\widetilde{m}m\nu}(f_\nu^{-1}(x)) + i \sigma \tau g_{\widetilde{m}m\nu}(f_\nu^{-1}(x)) \\
&= s_{\widetilde{m}m}(z;\Theta),
\end{align*}
where the last line follows from the fact that $\sigma \tau = (\sign \Im z)^{\widetilde{m}m}$.
\end{proof}

We illustrate Theorem~\ref{thm:comp} 
in Figure~\ref{fig:comp}.
\begin{figure}[htbp]
  \centering
\includegraphics[scale=0.45,trim=0.5in 2.5in 1in 2.5in,clip=true]{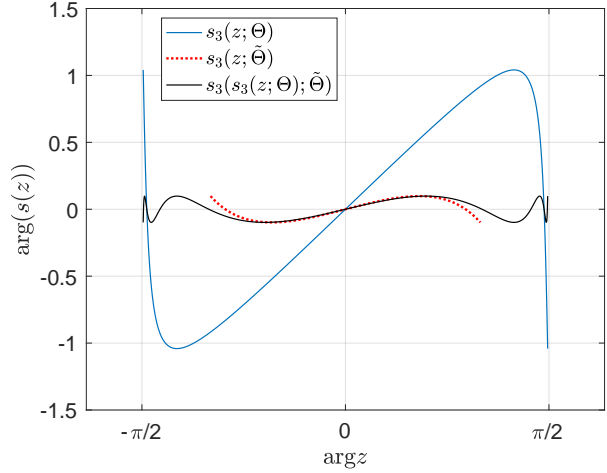}  
  \caption{Illustration of Theorem~\ref{thm:comp}
for $m=\widetilde{m}=3$, $\Theta=\pi/2-0.01$:
$s_{3}(z;\Theta)$, $s_{3}(z;\widetilde{\Theta})$ and 
$s_{3}(s_{3}(z;\Theta);\widetilde{\Theta}) = s_{9}(z;\Theta)$. 
Only $[-\Theta,\Theta]$ is shown; by symmetry 
the plots look the same on $[\pi-\Theta,\pi+\Theta]$.
Composing low-degree solutions results in a high-degree solution.}
  \label{fig:comp}
\end{figure}

\begin{remark}
Theorem~\ref{thm:comp} can also be proved by counting equioscillation points.  As $\theta$ runs from $-\Theta$ to $\Theta$, the number $\widetilde{\theta} := \arg \left( s_{m}(e^{i\theta};\Theta) \right)$ equioscillates $m+1$ times, taking values in $[-\widetilde{\Theta},\widetilde{\Theta}]$ and achieving its extrema at the endpoints.  Each time $\widetilde{\theta}$ runs from $\pm \widetilde{\Theta}$ to $\mp \widetilde{\Theta}$, the number 
\[
\widehat{\theta} := \arg \left( s_{\widetilde{m}}(e^{i\widetilde{\theta}};\widetilde{\Theta}) \right) = \arg \left( s_{\widetilde{m}}(s_{m}(e^{i\theta};\Theta);\widetilde{\Theta}) \right)
\]
equioscillates $\widetilde{m}+1$ times, achieving its extrema at the endpoints.  By counting extrema, we see that as $\theta$ runs from $-\Theta$ to $\Theta$, $\widehat{\theta}$ equioscillates $\widetilde{m} m+1$ times.
Since $s_{\widetilde{m}}(s_{m}(z;\Theta))$ is a rational function of degree $\widetilde{m} m$, we can argue as we did in the proof of Theorem~\ref{thm:main} that $s_{\widetilde{m}}(s_{m}(z;\Theta))$ must be a solution of Problem~\ref{problemZ6}.   Hence, $s_{\widetilde{m}}(s_{m}(z;\Theta);\widetilde{\Theta}) = s_{\widetilde{m}m}(z;\Theta)^\sigma$ for some $\sigma \in \{-1,1\}$.  Evaluating both sides of this equation at $z=i$ yields $\sigma=1$, so $s_{\widetilde{m}}(s_{m}(z;\Theta);\widetilde{\Theta}) = s_{\widetilde{m}m}(z;\Theta)$. 
\end{remark}

\begin{remark}
The identity~(\ref{fcomp}) shows that Zolotarev's (scaled) minimax approximant $F_m(x;\ell)$ of $\sign(x)$ on $[-1,-\ell] \cup [\ell,1]$ satisfies 
\begin{equation} \label{realsigncomp}
F_{\widetilde{m}}(F_{m}(x;\ell); \widetilde{\ell}) = F_{\widetilde{m}m}(x;\ell)
\end{equation}
whenever $\widetilde{\ell} = F_m(\ell;\ell)$.  
This composition law has been studied in, for example,~\cite{mysirev,bogatyrev2010chebyshev,bogatyrev2012rational}.
\end{remark} 
\begin{remark}
It is not hard to check that the function $\widetilde{s}_{2n+1}(z;\Theta) := s_{2n+1}(z;\Theta)^{(-1)^n}$ also behaves nicely under composition:  If $\widetilde{\Theta} = \left| \arg(\widetilde{s}_{2n+1}(e^{i\Theta}; \Theta)) \right|$, then
\[
\widetilde{s}_{2\widetilde{n}+1}(\widetilde{s}_{2n+1}(z;\Theta);\widetilde{\Theta}) = \widetilde{s}_{(2\widetilde{n}+1)(2n+1)}(z;\Theta).
\]
\end{remark}

Since
\[
\widetilde{s}_{2n+1}(z;\Theta) = \frac{z}{r_n(z^2; \Theta)},
\]
we obtain from Theorem~\ref{thm:comp} an analogous composition law for solutions of Problem~\ref{problemZ5}.

\begin{corollary}
Let $\Theta \in (0,\pi/2)$, $\widetilde{n},n \in \mathbb{N}_0$, and $\widetilde{\Theta} = \left| \arg(s_{2n+1}(e^{i\Theta}; \Theta)) \right| = \left|\arg\left( e^{i\Theta}/r_n(e^{2i\Theta};\Theta) \right)\right|$.  Then
\begin{equation} \label{sqrtcomp}
r_n(z;\Theta) r_{\widetilde{n}}\left( \frac{z}{r_n(z;\Theta)^2}; \widetilde{\Theta} \right) = r_{2\widetilde{n}n+\widetilde{n}+n}(z;\Theta).
\end{equation}
\end{corollary}
\begin{remark}
This behavior closely parallels the behavior of rational minimax approximants of $\sqrt{x}$ on positive real intervals; see~\cite{gawlik2018zolotarev,gawlik2018pth}. 
\end{remark}
\begin{proof}
We have
\begin{align*}
\frac{\sqrt{z}}{r_{2\widetilde{n}n+\widetilde{n}+n}(z;\Theta)}
&= \widetilde{s}_{4\widetilde{n}n+2\widetilde{n}+2n+1}(\sqrt{z};\Theta) \\
&= \widetilde{s}_{2\widetilde{n}+1}(\widetilde{s}_{2n+1}(\sqrt{z};\Theta); \widetilde{\Theta}) \\
&= \widetilde{s}_{2\widetilde{n}+1}\left( \frac{\sqrt{z}}{r_n(z;\Theta)}; \widetilde{\Theta} \right) \\
&= \frac{ \sqrt{z} / r_n(z;\Theta) }{ r_{\widetilde{n}}\left( z/r_n(z;\Theta)^2; \widetilde{\Theta} \right)}.
\end{align*}
Rearranging this yields~(\ref{sqrtcomp}).
\end{proof}

\subsection{Connections with other functions}

We conclude this section by drawing a few connections between the solutions to Problems~\ref{problemZ5}-\ref{problemZ6} and other well-studied functions.

\paragraph{\textbf{Finite Blaschke products}}
Ng and Tsang~\cite{ng2013polynomials,ng2015chebyshev} study a finite Blaschke product that behaves nicely under composition and solves the extremal problem~(\ref{zolonumber}) for $Z_m(E,F)$ with $E=[-\sqrt{\ell},\sqrt{\ell}]$ and $F=(-\infty,-\frac{1}{\sqrt{\ell}}] \cup [\frac{1}{\sqrt{\ell}},\infty)$.  The function is
\[
h_m(z;\ell) = \prod_{j=1}^m \frac{z-c_j}{1 - c_j z},
\]
where
\[
c_j = \frac{\sqrt{\ell} \cn\left( \frac{2j-1}{m}K(\ell), \ell \right)  }{ \dn\left( \frac{2j-1}{m}K(\ell), \ell \right) }.
\]
They show that if $\widetilde{\ell} = Z_m\big([-\sqrt{\ell},\sqrt{\ell}],(-\infty,-\frac{1}{\sqrt{\ell}}] \cup [\frac{1}{\sqrt{\ell}},\infty)\big)$, then~\cite[Proposition 2]{ng2013polynomials}
\[
h_{\widetilde{m}}(h_m(z;\ell);\widetilde{\ell}) = h_{\widetilde{m}m}(z;\ell)
\]
for any positive integers $\widetilde{m}$ and $m$, and~\cite[Proposition 4.1(b)]{ng2015chebyshev}
\[
\left(\frac{1-\widetilde{\ell}}{1+\widetilde{\ell}}\right) \frac{h_m(z;\ell)-1}{h_m(z;\ell)+1} = \frac{2}{1+F_m(\kappa;\kappa)} F_m(x;\kappa), 
\]
where
\[
x = \kappa \left(\frac{1+\sqrt{\ell}}{1-\sqrt{\ell}}\right) \frac{z-1}{z+1}, \quad \kappa = \left( \frac{1-\sqrt{\ell}}{1+\sqrt{\ell}} \right)^2.
\]
Our function $s_m$ is thus related to theirs via
\[
\left(\frac{1-\widetilde{\ell}}{1+\widetilde{\ell}}\right) \frac{h_m(z;\ell)-1}{h_m(z;\ell)+1} = \frac{1}{1+F_m(\kappa;\kappa)}\left( s_m(w;\Phi) + s_m(w;\Phi)^{-1} \right),
\]
where
\[
\frac{1}{2}(w+w^{-1}) = \left(\frac{1-\sqrt{\ell}}{1+\sqrt{\ell}}\right) \frac{z-1}{z+1}, \quad \cos\Phi = \left( \frac{1-\sqrt{\ell}}{1+\sqrt{\ell}} \right)^2 = \kappa.
\]

\paragraph{\textbf{Pad\'e approximants}}
In the limit as $\Theta \downarrow 0$, the solution to Problem~\ref{problemZ5} reduces to a Pad\'e approximant of $\sqrt{z}$.  
More precisely, let $p_n(z)$ denote the type-$(n,n)$ Pad\'e approximant to $\sqrt{z}$ at $z=1$.  An explicit formula for $p_n(z)$ is~\cite[p. 707]{gawlik2018zolotarev}
\[
p_n(z) = \sqrt{z} \frac{(1+\sqrt{z})^{2n+1} + (1-\sqrt{z})^{2n+1}}{(1+\sqrt{z})^{2n+1} - (1-\sqrt{z})^{2n+1}}.
\]
We say that a parametrized family of rational functions $r(z;\Theta)$ converges coefficientwise to $p_n(z)$ as $\Theta \downarrow 0$ if the coefficients in the numerator and denominator of $r(z;\Theta)$, appropriately normalized, converge to those of $p_n(z)$ as $\Theta \downarrow 0$.
\begin{proposition}
Let $n \in \mathbb{N}_0$.  As $\Theta \downarrow 0$, $r_n(z;\Theta)$ converges coefficientwise to $p_n(z)$.
\end{proposition}
\begin{proof}
Since $|r_n(z;\Theta)|=|p_n(z)|=1$ for all $z$ with $|z|=1$, it suffices to show that the poles of $r_n(z;\Theta)$ approach the poles of $p_n(z)$ as $\Theta \downarrow 0$.  It is easy to check that the poles of $p_n(z)$ are $\left\{ -\tan^2\left(\frac{j\pi}{2n+1}\right) \right\}_{j=1}^n$.  On the other hand, the poles of $r_n(z;\Theta)$ are $\{-a_j\}_{j=1}^n$.  Since $\lim_{\Theta\downarrow 0} K(\Theta) = K(0)= \pi/2$, $\lim_{\ell'\downarrow 0} \sn(z,\ell') =\sn(z,0)=\sin z$, $\lim_{\ell' \downarrow 0} \cn(z;\ell') =\cn(z,0)=\cos z$, and $\lim_{\ell' \downarrow 0} \dn(z,\ell') = \dn(z,0)=1$~\cite[Table 22.5.3]{NIST}, we have
\[
\lim_{\Theta \downarrow 0} a_j = \left( \frac{ \sin(2j-1)\omega + 1 }{ \cos(2j-1)\omega } \right)^{2(-1)^{j+n}},
\]
where $\omega = \pi/(4n+2)$.
Using the identities $\frac{\sin\theta+1}{\cos\theta} = \cot\left(\frac{\pi}{4}-\frac{\theta}{2}\right)$ and $\cot\left(\frac{\pi}{2}-\theta\right)=\tan\theta$, this can be simplified to
\begin{align*}
\lim_{\Theta \downarrow 0} a_j
&= \left( \cot (n-j+1)\omega \right)^{2(-1)^{j+n}} \\
&=
\begin{cases}
\tan^2(n-j+1)\omega, &\mbox{ if $j+n$ is odd}, \\
\tan^2(n+j)\omega, &\mbox{ if $j+n$ is even}.
\end{cases}
\end{align*}
This shows that $\{ \lim_{\Theta \downarrow 0} a_j \}_{j=1}^n$ contains the squared tangent of every even multiple of $\omega$.  Hence,
\[
\{ -\lim_{\Theta \downarrow 0} a_j \}_{j=1}^n = \left\{ -\tan^2\left(\frac{j\pi}{2n+1}\right) \right\}_{j=1}^n.
\]
\end{proof}

\paragraph{\textbf{Chebyshev polynomials}}
It is interesting to note the similarity between the results in this paper and the defining property of the Chebyshev polynomials of the first kind $T_n(x)$:
\[
\Re(z^n) = T_n(\Re z), \quad \text{ if } |z|=1.
\]
In fact, we can write Theorem~\ref{thm:rationalelliptic} in a more suggestive way by denoting 
\begin{align*}
\mathfrak{F}_m \colon [-1,1] \times (0,1) &\rightarrow [-1,1] \times (0,1), \\
(x,\ell) &\mapsto \big( F_m(x;\ell), F_m(\ell;\ell) \big),
\end{align*}
\begin{align*}
\mathfrak{s}_m \colon \mathbb{S} \times \mathbb{S}_+ &\rightarrow \mathbb{S} \times \mathbb{S}_+, \\
(z,\zeta) &\mapsto \big( s_m(z;|\arg\zeta|), s_m(\zeta;|\arg\zeta|) \big),
\end{align*}
and
\begin{align*}
J \colon \mathbb{S} \times \mathbb{S}_+ &\rightarrow [-1,1] \times (0,1), \\
(z,\zeta) &\mapsto ( \Re{z}, \Re{\zeta} ),
\end{align*}
where $\mathbb{S}=\{z \in \mathbb{C} \mid |z|=1\}$ and $\mathbb{S}_+=\{z \in \mathbb{S} \mid 0 < \Re z < 1\}$. 
With this notation, Theorem~\ref{thm:rationalelliptic} says that
\begin{equation} \label{FJJf}
\mathfrak{F}_m \circ J = J \circ \mathfrak{s}_m,
\end{equation}
and Theorem~\ref{thm:comp} says that
\begin{equation} \label{ff}
\mathfrak{s}_{\widetilde{m}} \circ \mathfrak{s}_m = \mathfrak{s}_{\widetilde{m}m}.
\end{equation}
By combining~(\ref{FJJf}) with~(\ref{ff}), we deduce
\begin{equation} \label{FF}
\mathfrak{F}_{\widetilde{m}} \circ \mathfrak{F}_m = \mathfrak{F}_{\widetilde{m}m},
\end{equation}
which is a restatement of~(\ref{realsigncomp}).
The identities~(\ref{FJJf}-\ref{FF}) mimic the following identities involving the monomials $t_n(z)=z^n$ and the Chebyshev polynomials $T_n(x)$:
\[
\left. T_n \circ \Re{} \right|_{\mathbb{S}} = \left. \Re{} \circ t_n\right|_{\mathbb{S}}, \quad t_m \circ t_n = t_{mn}, \quad T_m \circ T_n = T_{mn}.
\]

\bibliographystyle{amsplain2}
\bibliography{bib2}

\end{document}